\documentclass[11pt]{article}
 \usepackage{amsmath,amssymb}
 \usepackage{epsfig}
 \usepackage{graphicx}
 \usepackage{pict2e}
\usepackage{verbatim}
 \newtheorem{Lemma}{Lemma}
 \newtheorem{Proposition}[Lemma]{Proposition}
 
 \newtheorem{Theorem}[Lemma]{Theorem}
 \newtheorem{Conjecture}[Lemma]{Conjecture}

 \newtheorem{OP}[Lemma]{Open Problem}

 \newcommand{\GG}{\mbox{${\mathcal G}$}}

  \renewcommand{\SS}{\mbox{${\mathcal S}$}}

   \newcommand{\qed}{\ \ \rule{1ex}{1ex}}
    \newcommand{\eps}{\varepsilon}

\def\Ex{\mathbb{E}}
\renewcommand{\Pr}{\mathbb{P}}

\newcommand{\var}{\mathrm{var\ }}
  \newcommand{\bt}{{\mathbf t}}
   \newcommand{\bT}{{\mathbf T}}
   \newcommand{\bL}{{\mathbf L}}
 \newcommand{\ind}{{\rm 1\hspace{-0.90ex}1}}
 \begin{document}
 
\title{On the Largest Common Subtree of Random Leaf-Labeled Binary Trees}
 \author{David J. Aldous\thanks{Department of Statistics,
 367 Evans Hall \#\  3860,
 U.C. Berkeley CA 94720;  aldous@stat.berkeley.edu;
  www.stat.berkeley.edu/users/aldous.}
}

 \maketitle
 
 
 \begin{abstract}
 The size of the largest common subtree (maximum agreement subtree) of two independent uniform random binary trees on $n$ leaves is known to be between orders
 $n^{1/8}$ and $n^{1/2}$.
 By a construction based on recursive splitting and analyzable by standard ``stochastic fragmentation" methods, we improve the lower bound 
 to order $n^\beta$ for  $\beta = \frac{\sqrt{3} - 1}{2} = 0.366$.
 Improving the upper bound remains a challenging problem.
 \end{abstract}

\section{Introduction}
{\em Probabilistic combinatorics} is the study of random discrete structures -- such as graphs, trees, permutations and many more sophisticated structures.
This paper concerns leaf-labeled binary trees, illustrated in  Figure \ref{F2}. 
Like other models of random trees, this model has been studied for its interest to mathematicians, 
but this particular type of tree is also central to mathematical phylogenetics \cite{steel1,steel2} and, 
even though real-world phylogenetic tree data does not fit any simple model well 
(see e.g. \cite{me-yule,blum,purvis,xue}), the mathematical properties of the uniform random tree have also been studied within that literature.

The foundational fact is that, because the general $n+1$-leaf such tree is constructed uniquely by attaching an edge (to a new leaf labeled $n+1$) 
at a new branchpoint within one of the $2n - 3$ edges of an $n$-leaf tree,  the number $c_n$ of $n$-leaf trees must satisfy 
$c_{n+1} = (2n-3)c_n$ and so\footnote{This formula refers to {\em unrooted} trees, but the analog for rooted trees follows immediately.}
\begin{equation}
c_n = (2n-5)!! := (2n-5)(2n-7) \cdot \cdot 3 \cdot 1 .
\label{cn}
\end{equation}
Formula (\ref{cn})  prompts comparisons with  permutations. 
There has been extensive mathematical study of many aspects of the uniform random {\em permutation} on $n$ elements, 
so maybe there are analogous aspects of the uniform random leaf-labeled tree that are interesting to study.
In fact the problem we study is analogous to the well-studied {\em longest increasing subsequence}  (LIS) problem \cite{romik} 
for random permutations.\footnote{And to a broader range of {\em largest common substructure} problems -- see section \ref{sec:LCS}.}
We discuss similarities and differences in section \ref{sec:LIS}, but alas the deep theory associated with that problem 
(see \cite{corwin}  for a brief overview) does not seem helpful for our problem.

A {\em leaf-labeled binary tree} on $n$ leaves has $n-2$ {\em branchpoints} (degree three internal vertices) and $n$ distinct leaf labels; 
unless otherwise stated, by default the label-set is $[n] := \{1,2,\ldots,n\}$.  
We will generally write {\em tree} instead of  {\em leaf-labeled binary tree}. 
Note that a tree on $n$ leaves has $2n-3$ edges; and we call $n$ the {\em size} of the tree.
Figure \ref{F2} uses one way to draw a tree, though other ways are useful in other contexts.
Note there is no notion of {\em ordered}: the two trees on the right of Figure \ref{F2} are the same.

\setlength{\unitlength}{0.28in}
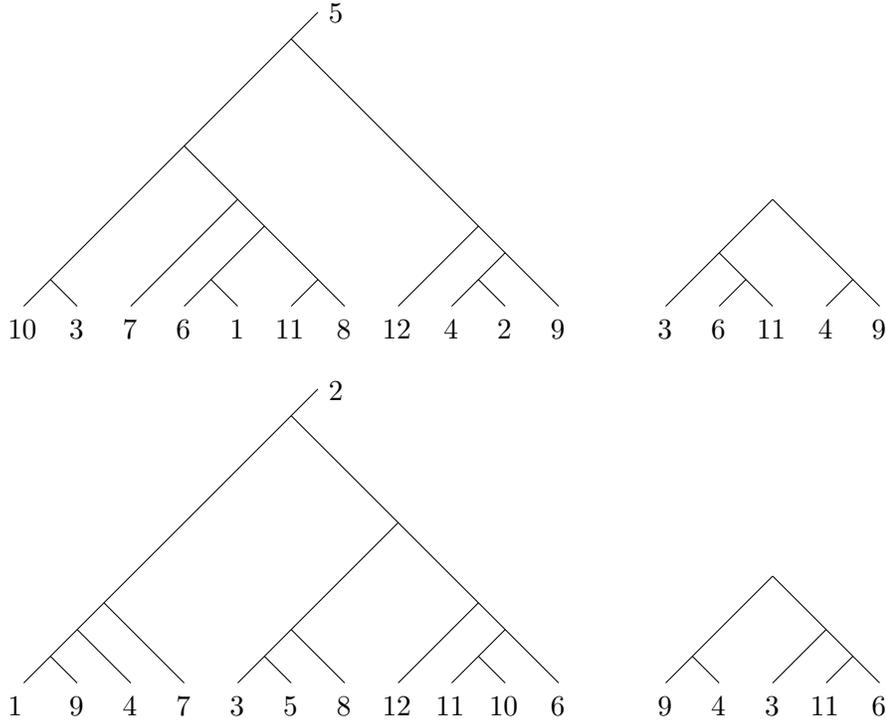
\begin{figure}[h!]
\caption{The two different larger trees (left side) have a common subtree, shown in two different representations on the right side.}
\label{F2}
\begin{picture}(17,7)(2,-1)
\put(2,0){\line(1,1){5.5}}
\put(3,0){\line(-1,1){0.5}}
\put(4,0){\line(1,1){2}}
\put(5,0){\line(1,1){1.5}}
\put(6,0){\line(-1,1){0.5}}
\put(7,0){\line(1,1){0.5}}
\put(8,0){\line(-1,1){3.0}}
\put(9,0){\line(1,1){1.5}}
\put(10,0){\line(1,1){1}}
\put(11,0){\line(-1,1){0.5}}
\put(12,0){\line(-1,1){5.0}}
\put(7.7,5.3){$5$}
\put(1.7,-0.6){$10$}
\put(2.85,-0.6){$3$}
\put(3.85,-0.6){$7$}
\put(4.85,-0.6){$6$}
\put(5.85,-0.6){$1$}
\put(6.7,-0.6){$11$}
\put(7.85,-0.6){$8$}
\put(8.7,-0.6){$12$}
\put(9.85,-0.6){$4$}
\put(10.85,-0.6){$2$}
\put(11.85,-0.6){$9$}
\put(13.85,-0.6){$3$}
\put(14.85,-0.6){$6$}
\put(15.7,-0.6){$11$}
\put(16.85,-0.6){$4$}
\put(17.85,-0.6){$9$}
\put(14,0){\line(1,1){2.0}}
\put(15,0){\line(1,1){0.5}}
\put(16,0){\line(-1,1){1.0}}
\put(17,0){\line(1,1){0.5}}
\put(18,0){\line(-1,1){2.0}}
\end{picture}

\begin{picture}(17,7)(2,-1)
\put(2,0){\line(1,1){5.5}}
\put(3,0){\line(-1,1){0.5}}
\put(4,0){\line(-1,1){1}}
\put(5,0){\line(-1,1){1.5}}
\put(6,0){\line(1,1){3}}
\put(7,0){\line(-1,1){0.5}}
\put(8,0){\line(-1,1){1}}
\put(9,0){\line(1,1){1.5}}
\put(10,0){\line(1,1){1}}
\put(11,0){\line(-1,1){0.5}}
\put(12,0){\line(-1,1){5.0}}
\put(7.7,5.3){$2$}
\put(1.7,-0.6){$1$}
\put(2.85,-0.6){$9$}
\put(3.85,-0.6){$4$}
\put(4.85,-0.6){$7$}
\put(5.85,-0.6){$3$}
\put(6.85,-0.6){$5$}
\put(7.85,-0.6){$8$}
\put(8.7,-0.6){$12$}
\put(9.7,-0.6){$11$}
\put(10.7,-0.6){$10$}
\put(11.85,-0.6){$6$}
\put(13.85,-0.6){$9$}
\put(14.85,-0.6){$4$}
\put(15.85,-0.6){$3$}
\put(16.7,-0.6){$11$}
\put(17.85,-0.6){$6$}
\put(14,0){\line(1,1){2.0}}
\put(15,0){\line(-1,1){0.5}}
\put(16,0){\line(1,1){1.0}}
\put(17,0){\line(1,1){0.5}}
\put(18,0){\line(-1,1){2.0}}
\end{picture}
\end{figure}

Any subset $A$ of leaves of a tree defines an {\em induced subtree} on leaf-set $A$, 
defined by first taking the spanning subtree and then deleting degree-2 internal vertices and joining edges to obtain a binary tree.
Given two trees $\bt$ and $\bt^\prime$, if there is a set $A$ of leaf-labels such that  the induced subtree on $A$ within $\bt$ 
is the same as the induced subtree on $A$ within $\bt^\prime$, call this a {\em common subtree}. 
See Figure \ref{F2}.
Now define 
$\kappa(\bt,\bt^\prime)$
to be the size of a maximum common subtree, in other words the size of a common subtree of maximum size.
This article studies the question
\begin{quote}
What can we say about $\kappa(\bT_n,\bT^\prime_n)$, where $\bT_n$ and $\bT^\prime_n$ are random, independent and uniform over 
trees with leaf-set $[n]$?
\end{quote}
This question (for the uniform model and some other models) has already been considered in several papers\footnote{Under the name {\em maximum agreement subtree}.}, most recently in 
\cite{bernstein,misra}, and the relevant known results\footnote{There are also extremal (worst-case) results: see \cite{extremal} for recent work.}  are as follows.
\begin{itemize}
\item The order of magnitude of $\Ex \kappa(\bT_n,\bT^\prime_n)$ is at most $n^{1/2}$: this is just the first moment method (calculating the expected number of large common subtrees). 
This upper bound continues to hold for some more general probability models \cite{pittel}.
\item The order of magnitude of $\Ex \kappa(\bT_n,\bT^\prime_n)$ is at least $n^{1/8}$: this is shown in \cite{bernstein} by first showing one can find 
a common {\em caterpillar} graph with order $n^{1/4}$ vertices and then using known LIS results to find a common subtree within the caterpillar graph.
\item In the alternate model where the two trees have the same shape (given the first uniform random tree, obtain the second by making a uniform random permutation of leaf-labels), \cite{misra} shows that the order of magnitude is exactly $n^{1/2}$.
\end{itemize}
In this article we improve the lower bound in the uniform case, as follows.
\begin{Theorem}
\label{T:1}
$\Ex \kappa(\bT_n,\bT^\prime_n) = \Omega(n^\beta)$ for all $\beta < \beta_0 := \frac{\sqrt{3} - 1}{2} = 0.366...$.
\end{Theorem}
For the record we state
\begin{OP}
\label{OP:1}
Prove that $\Ex \kappa(\bT_n,\bT^\prime_n) = O(n^\beta)$ for some $\beta < 1/2$
\end{OP}
which we strongly believe via heuristics below. 
Presumably  $\Ex \kappa(\bT_n,\bT^\prime_n) \approx n^\beta$ for some $\beta$, but we see no heuristic to guess 
the value of $\beta$ nor any methodology to prove such a $\beta$ exists.

\subsection{Background heuristics and related methodology}
\label{sec:backheur}
An $n$-leaf tree  has a {\em centroid}, meaning a branchpoint from which each of the three branches 
has size (number of leaves) at most $n/2$. 
Given a tree, we can ``split at the centroid", making each branch into a separate tree.
Write $z_{(1)} \ge z_{(2)} \ge z_{(3)}$
for the sizes of the branches, in decreasing order.
One can now imagine (as suggested in \cite{me-OP}) a recursive construction of a common subtree of two given trees $(\bt,\bt^\prime)$, as follows.
\begin{quote}
Consider as above the branch sizes $(z_{(i)}, z^\prime_{(i)}, i = 1,2,3)$ at the centroid of each tree, 
take for each $i$  the recursively-constructed common subtree for the two $i$'th branches, and then join these three common subtrees into
one common subtree of the two original trees.  
\end{quote}
The resulting tree will not be optimal, but analysis of its size will provide a lower bound on $\Ex \kappa(\bT_n,\bT^\prime_n)$,
and intuitively one expects its size to be close to optimal.  
Moreover its size scaling $n^\beta$ can heuristically be calculated as follows.
In the random tree setting, the number of labels in common between the $i$'th branch of the two trees is around $n X_{(i)}X^\prime_{(i)}$, 
where $X_{(i)} := n^{-1} Z_{(i)}$ is the normalized branch size.
So the common subtree within branch $i$ should have size about $( n X_{(i)}X^\prime_{(i)} )^\beta$,
leading heuristically to the relation
\begin{equation}
 1 = \Ex \sum_i (X_{(i)}X^\prime_{(i)})^\beta 
 \label{eq:heur}
 \end{equation}
in terms of the $n \to \infty$ limits of normalized branch sizes at the centroids. 
And one could solve (\ref{eq:heur}) for $\beta$ because the distributions of $X_{(i)}$ are known \cite{me-recursive}.

Alas this heuristic is not quite correct.
As explained in section \ref{sec:whynot}, 
we do not have a ``true recursion" involving different sizes of exactly the same structure, 
and this prevents us from obtaining a lower bound, though we could derive an upper bound on the size of the common subtree obtained in this way.

Note  that in the alternate ``same shape" model we have $X_{(i)} = X^\prime_{(i)}$ and so the heuristic  (\ref{eq:heur}) gives $\beta = 1/2$, consistent
with  the known result in that model.

\paragraph{Related methodology.}
The specific methodology of this paper -- analyzing some property of random trees by recursion at a centroid -- has apparently
not been used  before, except in \cite{me-recursive,me-tri} to study easier aspects of this model.
But it fits into the very broad area of  {\em divide and conquer} methods  (see e.g. \cite{parberry}).
It is very classical to study {\em rooted} trees by the decomposition into branches at the root: 
for instance the analysis of random branching processes \cite{athreya} or the study of randomized algorithms 
such as Quicksort that build a random tree 
 \cite{rosler}.
 Our specific use, to recursively construct a sub-optimal instance within an optimization problem, has been widely used for 
 analysis of Euclidean TSP-like problems on random points \cite{yukich}, via constructing an instance  on a large square by joining up instances on subsquares.
 Our use is more subtle in that our ``instances" only involve small subsets of the given elements; this only works because of the very special 
 ``consistency under sampling" property (section \ref{sec:key})  of our model.

\subsection{Outline proof of Theorem \ref{T:1}}
\label{sec:outline}
We will prove Theorem \ref{T:1} by studying a construction of a common subtree.
We use the same basic idea as in the heuristic above, but instead of the true centroid we decompose at 
a ``random centroid" defined via the branchpoint of the subtree induced by three uniform random vertices, chosen independently in each tree.
This structure does allow a true recursion, to be specified in detail in section \ref{sec:cons}.
It is known that the limit normalized branch sizes $(Y_1,Y_2,Y_3)$ now have the  Dirichlet($1/2,1/2,1/2$) distribution,
and the analog of (\ref{eq:heur}) is
 \begin{equation}
 1 = \Ex \sum_i (Y_iY^\prime_i)^\beta 
 \label{eq:beta}
 \end{equation}
 where  $(Y^\prime_1,Y^\prime_2,Y^\prime_3)$ is independent of  $(Y_1,Y_2,Y_3)$. 
 A standard fact  \cite{dirichlet} is that 
 $\Ex Y_i^\beta =  \frac{1}{1 + 2\beta}$,
 and equation (\ref{eq:beta}) implies $\Ex Y_i^\beta = \sqrt{1/3}$, so we can write the solution of (\ref{eq:beta}) as
 $\beta_0 = (\sqrt{3} -1)/2$.

 \section{Preliminaries}
 It seems helpful to highlight two preliminary results before we give the detailed description and analysis of the construction.

 \subsection{A key calculation}
 \label{sec:key}
Proposition \ref{P:key} below is the key (intricate but technically elementary) calculation that  involves the recursive decomposition and its probabilistic analysis.
 Recall that the number of $n$-leaf unrooted trees is
\begin{equation}
 c_n = (2n-5)!! = \frac{(2n-5)!}{2^{n-3} (n-3)!} . 
 \label{def:cn}
 \end{equation}
 Using Stirling's formula we obtain
\begin{equation}
c_n \sim 2^{n - 3/2}	 e^{-n} n^{n-2} 
\label{eq:cns}
\end{equation}
and it is useful to record the consequence
\begin{eqnarray}
c_{n+2}/n! & \sim & \pi^{-1/2} 2^{n} n^{-1/2}   \label{c2s} .
\end{eqnarray}
A fundamental feature, easily checked,  of our tree model
is the {\em consistency property} of uniform random trees: 
 \begin{quote}
given leaf-sets $A \subset A^\prime$, if $\bT_{A^\prime}$ is a uniform random tree on leaf-set $A^\prime$, then
the subtree $\bT_A$ induced by $A$ is uniform on all trees with leaf-set $A$.
\end{quote}
If $A$ is a  {\em random} subset, independent of $\bT_{A^\prime}$, it remains true that (conditional on $A$) the induced subtree $\bT_A$ is uniform on all trees with leaf-set $A$.

Proposition \ref{P:key}, illustrated in Figure \ref{F:construct_1.5}, describes the  general step of our recursive construction.
 \begin{Proposition}
 \label{P:key}
 Consider a uniform random tree on leaf-set $A \cup \{b,b^*\}$, where $A \subseteq [n]$ has $|A| = m \ge 1$, and $b, b^*$ are labels not in $[n]$.
 Pick a uniform random element of $A$ and re-label it as $b^{**}$.
 Split the tree into 3 branches at the branchpoint of the induced subtree on $\{b,b^*,b^{**}\}$, and in each branch create a new leaf at the former branchpoint 
 and label these new leaves as $b_1, b_2, b_3$ according as the branch contains $b, b^*, b^{**}$. 
 So we obtain three random trees, say $\bT_{(1)}, \bT_{(2)}, \bT_{(3)}$, on leaf-sets of the form 
 $A_1 \cup \{b,b_1\}$ and 
  $A_2 \cup \{b^*,b_2\}$ and 
   $A_3 \cup \{b^{**},b_3\}$. 
   So $\sum_i  |A_i |= m - 1$.
 Then, for $c_n$ defined at (\ref{def:cn}) and for a non-negative triple $(m_1,m_2,m_3)$ with $m_1+m_2+m_3 = m-1$,
\[ (i)\ \  \Pr( |A_1| = m_1, \ |A_2| = m_2, |A_3| = m_3) = 
 \frac{ {m \choose m_1 \ m_2\  m_3\  1} \ c_{m_1+2}\  c_{m_2+2} \ c_{m_3+2} }{m \ c_{m+2}} .\]
 (ii) Conditional on $A_1, A_2, A_3$, the random trees $(\bT_{(i)}, 1 \le i \le 3)$ are independent and uniform on their respective leaf-sets.
 \end{Proposition}
 Note we may have $|A_i| = 0$, that is $A_i = \emptyset$.

 \setlength{\unitlength}{0.06in}
\begin{figure}
\caption{The natural bijection.}
\label{F:construct_1.5}
\begin{picture}(70,50)(7,-3)
\put(20,20){\line(-1,-1){19}}
\put(20,20){\line(1,-1){19}}
\put(20,20){\line(0,1){19}}
\put(20,20){\circle{44}}
\put(20,22){\circle*{1}}
\put(18,18){\circle*{1}}
\put(22,18){\circle*{1}}
\put(1,1){\circle*{1}}
\put(39,1){\circle*{1}}
\put(-1.5,0.5){$b$}
\put(40.0,0.5){$b^*$}
\put(20,39){\circle*{1}}
\put(21,38.5){$b^{**}$}

\put(3,25){$A$}

\put(20,28){\line (1,0){4}}
\put(24,28){\circle*{1}}
\put(20,32){\line (-1,0){4}}
\put(16,32){\circle*{1}}
\put(12,12){\line(-1,1){3}}
\put(9,15){\circle*{1}}
\put(9,9){\line(1,-1){3}}
\put(12,6){\circle*{1}}
\put(28,12){\line(1,1){3}}
\put(31,15){\circle*{1}}
\put(31,9){\line(-1,-1){3}}
\put(28,6){\circle*{1}}

\put(65,10){\circle{14}}
\put(85,10){\circle{14}}
\put(75,30){\circle{12}}
\put(75,22){\circle*{1}}
\put(73,18){\circle*{1}}
\put(77,18){\circle*{1}}
\put(73,18){\line(-1,-1){17}}
\put(77,18){\line(1,-1){17}}
\put(75,22){\line(0,1){17}}
\put(56,1){\circle*{1}}
\put(94,1){\circle*{1}}
\put(53,0.5){$b$}
\put(95.0,0.5){$b^*$}
\put(75,39){\circle*{1}}
\put(76,38.5){$b^{**}$}
\put(69.7,17.5){$b_1$}
\put(78,17.5){$b_2$}
\put(76,21.5){$b_3$}

\put(75,28){\line (1,0){4}}
\put(79,28){\circle*{1}}
\put(75,32){\line (-1,0){4}}
\put(71,32){\circle*{1}}
\put(67,12){\line(-1,1){3}}
\put(64,15){\circle*{1}}
\put(64,9){\line(1,-1){3}}
\put(67,6){\circle*{1}}
\put(83,12){\line(1,1){3}}
\put(86,15){\circle*{1}}
\put(86,9){\line(-1,-1){3}}
\put(83,6){\circle*{1}}

\put(59,11){$A_1$}
\put(88,11){$A_2$}
\put(77,31){$A_3$}

\end{picture}
\end{figure}
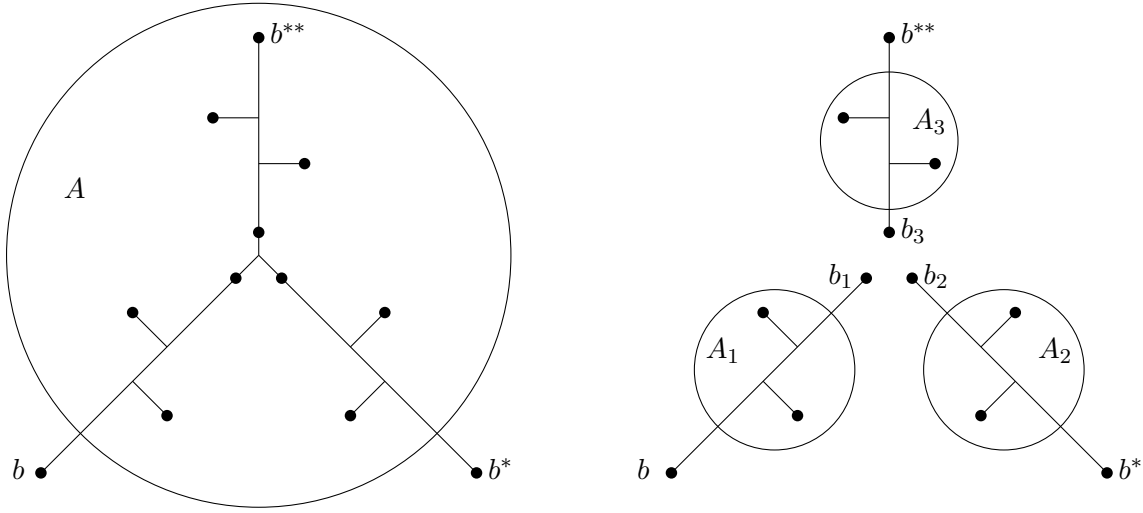
 
 \noindent
 \begin{proof}
 Consider $A \subseteq [n]$ and $b^{**} \in A$ and distinct labels $\{b, b^*, b_1, b_2, b_3\}$ not in $[n]$. 
 As illustrated in Figure \ref{F:construct_1.5}, the construction corresponds to the natural bijection between
 
 \noindent (a) trees on leaf-set $A \cup \{b,b^*\}$ with one leaf of $A$ re-labeled as $b^{**}$;
 
 \noindent (b) partitions of $A$ as $(A_1,A_2,A_3, \{b^{**}\})$ and trees on leaf-sets  $A_1 \cup \{b,b_1\}$ and 
  $A_2 \cup \{b^*,b_2\}$ and 
   $A_3 \cup \{b^{**},b_3\}$. 
   
  \noindent
 The number of elements in (a) equals $m c_{m+2}$, and the number of elements in (b) with $|A_i| = m_i \ \forall i$ equals
 ${m \choose m_1\  m_2 \ m_3 \ 1} c_{m_1+2}\  c_{m_2+2} \ c_{m_3+2} $.
 This establishes part (i).
 For (ii), given disjoint $(A_1,A_2,A_3)$ with $\sum_i |A_i| = m-1$, and given trees 
 $\bt_{(1)}, \bt_{(2)}, \bt_{(3)}$ on leaf-sets of the form 
 $A_1 \cup \{b,b_1\}$ and 
  $A_2 \cup \{b^*,b_2\}$ and 
   $A_3 \cup \{b^{**},b_3\}$,
   the bijection shows
   \[ \Pr (\bT_{(1)} = \bt_{(1)}, \bT_{(2)} = \bt_{(2)}, \bT_{(3)} = \bt_{(3)}) = \frac{1}{m c_m}  .\]
   The fact that, for given $(A_1,A_2,A_3)$, this is constant as a function of   $(\bt_{(1)}, \bt_{(2)}, \bt_{(3)})$,
   is equivalent to assertion (ii).
 \qed
 \end{proof}
 
 \medskip
 \noindent
 Writing $A_i(m)$ to show the dependence on $m$, 
and using (\ref{c2s}), we get the sharp asymptotics
\begin{eqnarray*}
 \lefteqn{\Pr( |A_1(m)| = m_1, \ |A_2(m)| = m_2, |A_3(m)| = m_3)}&& \\
  &\sim & (2 \pi)^{-1} m^{-2} \prod_{i=1}^3 (m_i/m)^{-1/2} 
 \mbox{ as } m, m_i \to \infty .
 \end{eqnarray*}
 This says that the normalized branch sizes converge in distribution to the Dirichlet(1/2,1/2,1/2) distribution, that is
 \begin{equation}
 (m^{-1}|A_1(m)|, m^{-1}|A_2(m)|, m^{-1}|A_3(m)| ) \to_d (Y_1,Y_2,Y_3)
 \label{Dlim}
 \end{equation}
 where the limit has density $(2 \pi)^{-1} y_1^{-1/2} y_2^{-1/2}  y_3^{-1/2} $ on the simplex 
 $\{y_1 + y_2 + y_3 = 1\}$.

\subsection{The associated martingale in the continuous limit}
Asymptotic results for finite random structures often\footnote{Indeed the preceding Dirichlet($1/2,1/2,1/2$) distribution arises as the exact distribution of branch masses within the 
Brownian continuum random tree, which is the scaling limit of various discrete random tree models \cite{evans} including ours.
See \cite{me-tri} for an alternate derivation of  formulas relating to (\ref{Dlim}). 
Alas this ``scaling limit" convergence is not informative enough for our ``common subtree" problem.} 
 correspond to exact results for some limit continuous process.
And indeed the heuristics in section \ref{sec:outline} can be related to  an exact result, Lemma \ref{L:MG} below,  for the following continuous fragmentation-type process.

\paragraph{A continuous model.}  
\label{sec:cont}
Take two independent random vectors $(Y_1,Y_2,Y_3)$ and $(Y^\prime_1,Y^\prime_2, Y^\prime_3)$, each with the
Dirichlet$(1/2,1/2,1/2)$ distribution on the  2-simplex.
Form another distribution on the 3-simplex by
\begin{equation}
\bL=  (L_1,L_2,L_3,L_0) := (Y_1Y^\prime_1, Y_2Y^\prime_2, Y_3Y^\prime_3, 1 - Y_1Y^\prime_1 - Y_2Y^\prime_2- Y_3Y^\prime_3) .
 \label{def:L}
 \end{equation}
As in some freshman probability textbook examples,  let's describe the process in terms of cookie dough and one chocolate chip.
We have a mass $u$ of dough containing the chip.
We divide the dough into 4 pieces according to $\bL$, that is of masses $(uL_1,uL_2,uL_3,uL_0)$.
The chip follows in the natural way, with chance $L_i$ to get into the lump of mass $uL_i$.
We then throw away the fourth piece, of mass $uL_0$, so maybe discarding the chip.
Continue recursively splitting and discarding pieces of dough, using independent realizations of $\bL$.   

Within this process, starting with $1$ unit of dough, we can define $Z(0) = 1$ and
\[ Z(t) = \mbox{ mass of the piece containing the chip after $t$ splits} \]
 where $Z(t) = 0$ if the chip has been discarded.
 
\begin{Lemma}
 \label{L:MG}
 Set $\beta_0 = (\sqrt{3} -1)/2$.  Then
 \begin{equation}
 \mbox{the process $((Z(t))^{\beta_0 - 1} \ind_{(Z(t) > 0)}, t \ge 0)$ is a martingale.}
 \label{MG}
 \end{equation}
 \end{Lemma}
 \begin{proof}
 By considering the first split,
 \[ \Ex  [ (Z(1))^{\beta - 1}  \ind_{(Z(1) > 0)} ]= \Ex [\sum_{i=1}^3 L_i (L_i)^{\beta - 1} ] \]
 and setting the right side equal to $1$ repeats the equation (\ref{eq:beta}) for $\beta$ whose solution was this value  $\beta_0$.
 This is the martingale property at $t = 1$, and the general case follows by scaling.
 \qed
 \end{proof}
 
 Within our model, the process $Z(t)$ arises as an approximation to the sizes of subtrees containing a given leaf at successive stages of the recursive construction.
 We formalize the relation in section \ref{sec:formal} in order to find the expected size of the common subtree produced by our scheme.

Similar martingale arguments are standard  in stochastic fragmentation models (see e.g. \cite{bertoin}),
though generally appear in  settings where mass is conserved.

\section{Proof of Theorem \ref{T:1}}
We emphasize two points about the construction.
It is a top-down construction via specifying branchpoints, and we do not specify vertices until near the end. 
And our implementation is based on random choices, which seems quite inefficient -- heuristically, the ``true centroid" scheme must work better,
that is produce a larger common subtree, because the matching of branches by size maximizes the number of vertices landing in the same branch.
 But random choices allow us to analyze the performance.
So what we describe first (section \ref{sec:cons}) is a 
randomized algorithm, which can be applied to two arbitrary trees $\bt$ and $\bt^\prime$ on leaf-set $[n]$ and which will always output  a (random)
common subtree.
Then in section \ref{sec:analysis} we prove a lower bound on the expectation of the size of the output tree, when the algorithm is applied to  two independent uniform random trees of size $n$.

\subsection{The construction}
\label{sec:cons}
The construction is  illustrated in Figures \ref{F:construct} -- \ref{F:construct_3}, and is easiest to understand via the pictures.
We are given two arbitrary trees $\bt$ and $\bt^\prime$ on leaf-set $[n]$.
For consistency with later stages, we first pick two distinct leaves uniformly at random in $\bt$, and replace their labels by labels $t_1, t_2$. 
Such {\em novel} labels are distinct from the {\em original} labels, which are now a subset $B$ of $[n]$ with $|B| = n-2$.
Repeat independently within $\bt^\prime$, using the same novel labels but typically\footnote{When $n$ is large.} deleting different original labels and so retaining a different subset $B^\prime$ of original labels.
Finally set $A = B \cap B^\prime$ and consider the induced subtrees on leaf-set $A \cup \{t_i, t_2\}$,
so we get a subtree $\bt_0$ within $\bt$ and a subtree $\bt^\prime_0$ within $\bt^\prime$.
This produces what we will call the Stage 0 configuration, which consists of two (random) trees on the {\em same} (random) leaf-set $A \cup \{t_1, t_2\}$,
where typically $|A| = n-4$.

\paragraph{Stage 1.}
In $\bt_0$, choose a third original leaf uniformly at random, and re-label it $t_3$.
The three leaves $t_1, t_2, t_3$ determine a branchpoint (of the induced subtree).
As illustrated in Figure \ref{F:construct_1.5},
cut $\bt_0$ into $3$ branches at the branchpoint, and within each branch (labeled $i = 1,2,3 $ according to $t_i$) create a new leaf labeled $b_i$ at the cut-point.
So branch $i$ contains {\em novel} leaves $b_i$ and $t_i$ and some subset $B_i \subset [n]$ of {\em original} leaves. 
Our notational convention is that 
$b$ indicates {\em branchpoint} and $t$ indicates {\em terminal}. These are slightly different in that a branchpoint may correspond to a branchpoint in the ultimate common subtree whereas a terminal cannot correspond to a leaf in the common subtree. 

Repeat independently within $\bt^\prime_0$. 
That is, within $\bt^\prime_0$ we obtain {\em novel} leaves which are also labeled $t_i$ and $b_i$, but now in $\bt^\prime$  the branches contain different subsets $B^\prime_i$ of  {\em original}  leaves.
For each $i$ define $A_i = B_i \cap B^\prime_i$ and assume each $A_i$ is non-empty.
For each $i$ and each of $\bt$ and $\bt^\prime$ consider the subtree within branch $i$ induced by the labels $A_i \cup \{b_i, t_i\}$; call these trees $\bt_{(i)}$ and $\bt^\prime_{(i)}$.  
This is the Stage 1 configuration.
Note that  $\bt_{(i)}$ and $\bt^\prime_{(i)}$ are (typically different) trees on the {\em same} leaf-set.
And each leaf-set has a specific structure:
\[ A \cup \{t^o,t^{oo}\} \mbox{ where } A \subset [n] \mbox{ and } t^o,t^{oo} \mbox{are two novel labels}.
\]
This of course is the structure of the Stage 0 configuration.
Figure \ref{F:construct} (top) is a summary of this construction within $\bt$: Figure \ref{F:construct_2} later illustrates two different trees on the same leaf-set.


\setlength{\unitlength}{0.08in}
\begin{figure}
\caption{The construction: initial steps.  Each ``A" represents the set of leaves of the two corresponding trees ``$\bt$" within the original trees.
The first stage creates subtrees $t_{(i)}$.
The second stage shows $t_{(1)}$ split into $t_{(11)}, t_{(12)}, t_{(13)}$.
 The third stage shows $t_{(13)}$ and $t_{(12)}$ split further.}
\label{F:construct}
\begin{picture}(70,100)(0,-5)
\put(48,80){\circle*{1}}
\put(46,78){$b_1$}
\put(30,80){\circle*{1}}
\put(27.5,79.5){$t_1$}
\put(51,82){\circle*{1}}
\put(52,81.5){$b_2$}
\put(65,94){\circle*{1}}
\put(66,93.5){$t_2$}
\put(51,78){\circle*{1}}
\put(52,77.5){$b_3$}
\put(65,66){\circle*{1}}
\put(66,65.5){$t_3$}
\multiput(51,82)(1.4,1.2){10}{\line(14,12){0.7}}
\multiput(51,78)(1.4,-1.2){10}{\line(14,-12){0.7}}
\multiput(48,80)(-2.0,0){9}{\line(-1,0){0.7}}
\put(36,81){$A_1$}
\put(36,78){$\bt_{(1)}$}
\put(56,89){$A_2$}
\put(59,87){$\bt_{(2)}$}
\put(56,70.5){$A_3$}
\put(59,72){$\bt_{(3)}$}
\put(35,88){STAGE 1}

\put(40,50){\circle*{1}}
\put(38.2,48.2){$b_1$}
\put(0,50){\circle*{1}}
\put(-2.5,49.5){$t_1$}
\put(18,54){\circle*{1}}
\put(16,52){$b_{12}$}
\put(22,54){\circle*{1}}
\put(22.7,52){$b_{11}$}
\put(20,56){\circle*{1}}
\put(20.8,56.8){$b_{13}$}
\put(20,68){\circle*{1}}
\put(17.1,67.5){$t_{11}$}
\put(43,52){\circle*{1}}
\put(44,51.5){$b_2$}
\put(43,48){\circle*{1}}
\put(44,47.5){$b_3$}
\multiput(43,52)(1.4,1.2){4}{\line(14,12){0.7}}
\multiput(43,48)(1.4,-1.2){4}{\line(14,-12){0.7}}
\multiput(22,54)(1.8,-0.4){10}{\line(18,-4){0.9}}
\multiput(18,54)(-1.8,-0.4){10}{\line(-18,-4){0.9}}
\multiput(20,56)(0,2){6}{\line(0,2){1.0}}
\put(31,53){$A_{11}$}
\put(31,49.8){$\bt_{(11)}$}
\put(6,53){$A_{12}$}
\put(6.5,50){$\bt_{(12)}$}
\put(20.5,61){$A_{13}$}
\put(16,61){$\bt_{(13)}$}
\put(4,60){STAGE 2}

\put(4,39){\circle*{1}}
\put(1,39.5){$t_{11}$}
\put(4,16){\circle*{1}}
\put(1,15.5){$b_{13}$}
\put(6,14){\circle*{1}}
\put(5.1,12){$b_{11}$}
\put(2,14){\circle*{1}}
\put(0.1,12){$b_{12}$}
\put(42,6){\circle*{1}}
\put(40.2,4.2){$b_1$}
\put(45,8){\circle*{1}}
\put(46,7.5){$b_2$}
\put(45,4){\circle*{1}}
\put(46,3.5){$b_3$}

\multiput(45,8)(1.4,1.2){4}{\line(14,12){0.7}}
\multiput(45,4)(1.4,-1.2){4}{\line(14,-12){0.7}}
\multiput(2,14)(-1.8,-0.4){4}{\line(-18,-4){0.9}}

\put(29,16){\circle*{1}}
\put(25.1,16.5){$b_{113}$}
\put(31,14){\circle*{1}}
\put(29.8,12){$b_{111}$}
\put(27,14){\circle*{1}}
\put(25.1,12){$b_{112}$}
\put(36,23){\circle*{1}}
\put(32.1,23.5){$t_{111}$}

\multiput(31,14)(1.65,-1.2){7}{\line(11,-8){0.55}}
\multiput(26.5,14)(-2,0){10}{\line(-1,-0){1}}
\multiput(29,16)(1.4,1.4){5}{\line(7,7){0.7}}

\put(14,31){\circle*{1}}
\put(10.1,31.1){$b_{132}$}
\put(16,29){\circle*{1}}
\put(15.1,27){$b_{133}$}
\put(12,29){\circle*{1}}
\put(10.1,27){$b_{131}$}
\put(31,38){\circle*{1}}
\put(32.1,38.5){$t_{131}$}

\multiput(14,31)(-1.5,1.2){7}{\line(-10,8){0.75}}
\multiput(16,29)(1.5,0.9){10}{\line(15,9){0.75}}
\multiput(12,29)(-0.8,-1.3){10}{\line(-8,-13){0.4}}

\put(37,10.2){$A_{111}$}
\put(34,8){$\bt_{(111)}$}
\put(32.5,18){$A_{113}$}
\put(29,20.5){$\bt_{(113)}$}
\put(16,15){$A_{112}$}
\put(16,12){$\bt_{(112)}$}

\put(8,21.5){$A_{131}$}
\put(3.0,21.5){$\bt_{(131)}$}
\put(21.5,35){$A_{133}$}
\put(24,32){$\bt_{(133)}$}
\put(10,35){$A_{132}$}
\put(5.6,34){$\bt_{(132)}$}

\put(36,30){STAGE 3}

\end{picture}
\end{figure}
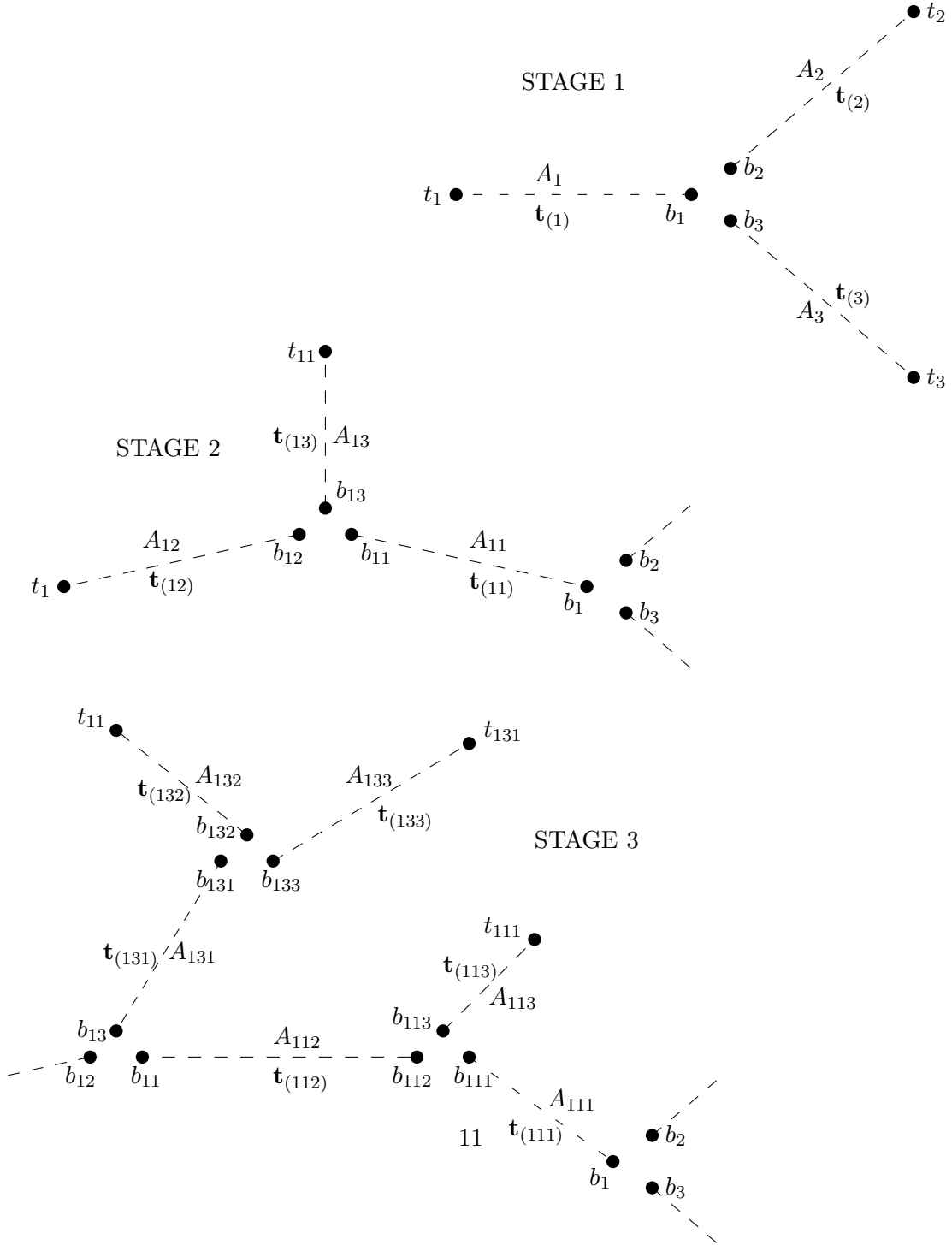

\setlength{\unitlength}{0.08in}
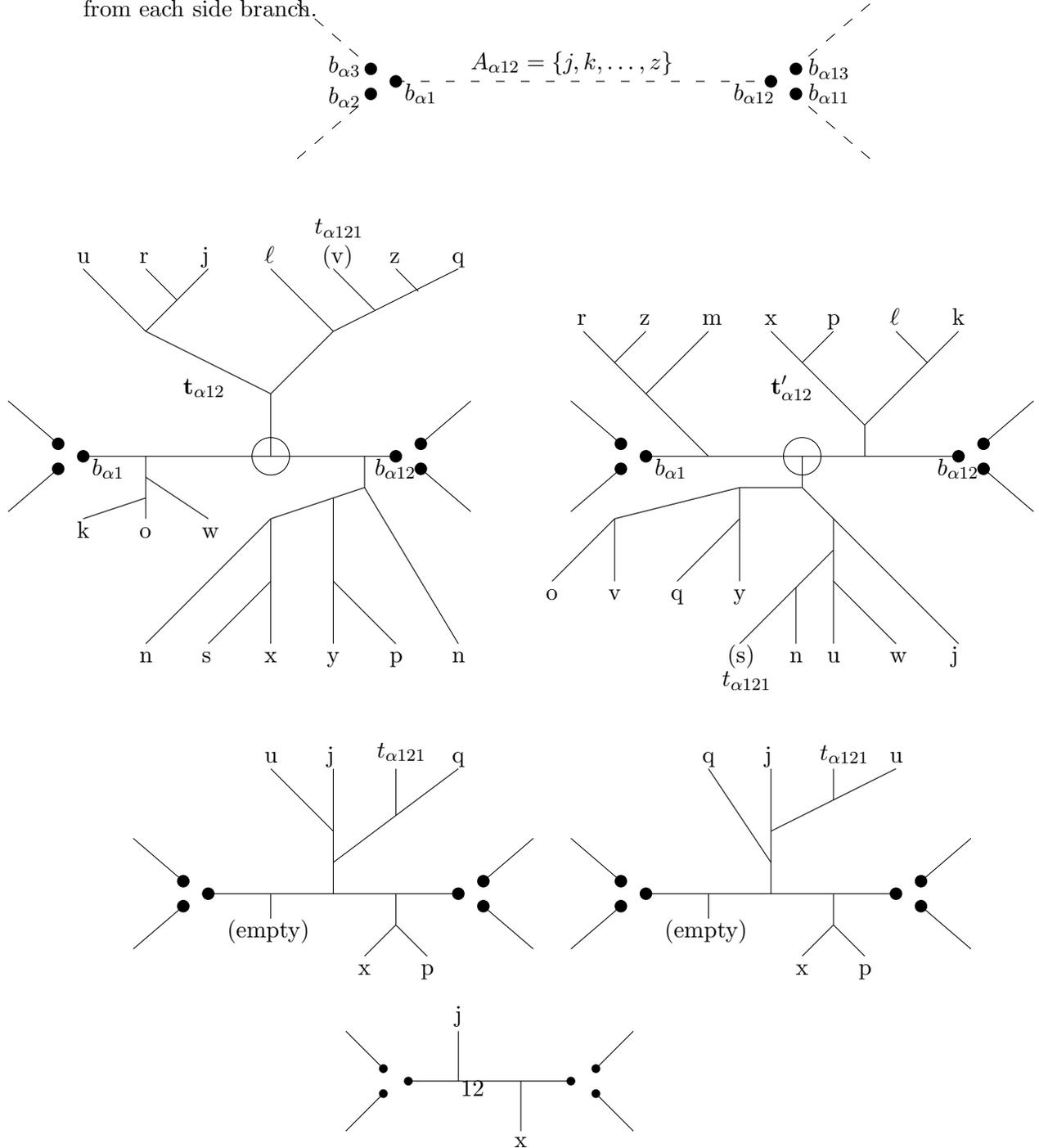
\begin{figure}
\caption{The construction: final split.
The top row shows a matched edge in the style of Figure \ref{F:construct}.
This is expanded in the second row to show the corresponding full side branches, 
and the third row contracts to show the induced subtrees on common vertices. 
The fourth row shows one leaf picked from each side branch.}
\label{F:construct_2}
\begin{picture}(70,110)(0,-10)
\put(55,95){\circle*{1}}
\put(57,96){\circle*{1}}
\put(57,94){\circle*{1}}
\multiput(25,95)(2.0,0){15}{\line(1,0){0.7}}
\put(25,95){\circle*{1}}
\put(23,96){\circle*{1}}
\put(23,94){\circle*{1}}
\multiput(58,97)(1.4,1.2){4}{\line(14,12){0.7}}
\multiput(58,93)(1.4,-1.2){4}{\line(14,-12){0.7}}
\multiput(22,97)(-1.4,1.2){4}{\line(-14,12){0.7}}
\multiput(22,93)(-1.4,-1.2){4}{\line(-14,-12){0.7}}
\put(58,95.5){$b_{\alpha 13}$}
\put(58,93.5){$b_{\alpha 11}$}
\put(52,93.5){$b_{\alpha 12}$}

\put(25.7,93.5){$b_{\alpha 1}$}
\put(19.6,93.2){$b_{\alpha 2}$}
\put(19.6,95.7){$b_{\alpha 3}$}

\put(31,96){$A_{\alpha 12} = \{j,k,\ldots,z\}$}


\put(0,65){\circle*{1}}
\put(-2,66){\circle*{1}}
\put(-2,64){\circle*{1}}
\put(0.7,63.5){$b_{\alpha 1}$}
\put(23.3,63.5){$b_{\alpha 12}$}
\put(25,65){\circle*{1}}
\put(27,66){\circle*{1}}
\put(27,64){\circle*{1}}
\put(0,65){\line(1,0){25}}

\put(-2,66){\line(-14,12){4}}
\put(-2,64){\line(-14,-12){4}}
\put(27,66){\line(14,12){4}}
\put(27,64){\line(14,-12){4}}

\put(0,80){\line(1,-1){5}}
\put(5,80){\line(1,-1){2,5}}
\put(10,80){\line(-1,-1){5}}
\put(15,80){\line(1,-1){5}}
\put(20,80){\line(1,-1){3.3}}
\put(25,80){\line(1,-1){1.8}}
\put(30,80){\line(-2,-1){10}}
\put(15,70){\line(0,-1){5}}
\put(15,70){\line(1,1){5}}
\put(15,70){\line(-2,1){10}}

\put(5,65){\line(0,-1){5}}
\put(0,60){\line(3,1){5}}
\put(10,60){\line(-3,2){5}}

\put(15,60){\line(0,-1){10}}
\put(15,60){\line(-1,-1){10}}
\put(15,55){\line(-1,-1){5}}

\put(22.5,62.5){\line(0,1){2.5}}
\put(22.5,62.5){\line(7.5,-12.5){7.5}}
\put(22.5,62.5){\line(-7.5,-2.5){7.5}}
\put(20,50){\line(0,1){11.67}}
\put(25,50){\line(-1,1){5}}

\put(-0.5,80.5){u}
\put(4.5,80.5){r}
\put(9.5,80.5){j}
\put(14.5,80.5){$\ell$}
\put(19.1,80.5){(v)}
\put(18.5,82.8){$t_{\alpha 121}$}
\put(24.5,80.5){z}
\put(29.5,80.5){q}

\put(-0.5,58.5){k}
\put(4.5,58.5){o}
\put(9.5,58.5){w}
\put(4.5,48.5){n}
\put(9.5,48.5){s}
\put(14.5,48.5){x}
\put(19.5,48.5){y}
\put(24.5,48.5){p}
\put(29.5,48.5){n}


\put(45,65){\circle*{1}}
\put(43,66){\circle*{1}}
\put(43,64){\circle*{1}}
\put(45.7,63.5){$b_{\alpha 1}$}
\put(68.3,63.5){$b_{\alpha 12}$}
\put(70,65){\circle*{1}}
\put(72,66){\circle*{1}}
\put(72,64){\circle*{1}}
\put(45,65){\line(1,0){25}}

\put(43,66){\line(-14,12){4}}
\put(43,64){\line(-14,-12){4}}

\put(72,66){\line(14,12){4}}
\put(72,64){\line(14,-12){4}}

\put(50,65){\line(-1,1){10}}
\put(45,75){\line(-1,-1){2.5}}
\put(50,75){\line(-1,-1){5}}

\put(62.5,67.5){\line(0,-1){2.5}}
\put(62.5,67.5){\line(1,1){7.5}}
\put(62.5,67.5){\line(-1,1){7.5}}
\put(57.5,72.5){\line(1,1){2.5}}
\put(67.5,72.5){\line(-1,1){2.5}}

\put(60,60){\line(0,-1){10}}
\put(60,60){\line(1,-1){10}}
\put(60,55){\line(1,-1){5}}
\put(60,57.5){\line(-1,-1){7.5}}
\put(57,50){\line(0,1){4.5}}
\put(57.5,62.5){\line(0,1){2.5}}
\put(57.5,62.5){\line(1,-1){2.5}}
\put(57.5,62.5){\line(-1,0){5}}

\put(52.5,55){\line(0,1){7.5}}
\put(47.5,55){\line(1,1){5}}
\put(42.5,55){\line(0,1){5}}
\put(37.5,55){\line(1,1){5}}
\put(42.5,60){\line(4,1){10}}

\put(39.5,75.5){r}
\put(44.5,75.5){z}
\put(49.5,75.5){m}
\put(54.5,75.5){x}
\put(59.5,75.5){p}
\put(64.5,75.5){$\ell$}
\put(69.5,75.5){k}

\put(37,53.5){o}
\put(42,53.5){v}
\put(47,53.5){q}
\put(52,53.5){y}
\put(51.5,48.5){(s)}
\put(51.1,46.6){$t_{\alpha 121}$}
\put(56.5,48.5){n}
\put(59.5,48.5){u}
\put(64.5,48.5){w}
\put(69.4,48.5){j}


\put(10,30){\circle*{1}}
\put(8,31){\circle*{1}}
\put(8,29){\circle*{1}}
\put(10,30){\line(1,0){20}}
\put(30,30){\circle*{1}}
\put(32,31){\circle*{1}}
\put(32,29){\circle*{1}}

\put(32,31){\line(14,12){4}}
\put(32,29){\line(14,-12){4}}
\put(8,31){\line(-14,12){4}}
\put(8,29){\line(-14,-12){4}}

\put(20,30){\line(0,1){10}}
\put(15,40){\line(1,-1){5}}
\put(30,40){\line(-4,-3){10}}
\put(25,40){\line(0,-1){3.75}}

\put(22.5,25){\line(1,1){2.5}}
\put(27.5,25){\line(-1,1){2.5}}
\put(25,30){\line(0,-1){2.5}}
\put(15,30){\line(0,-1){2}}
\put(11.5,26.5){(empty)}

\put(45,30){\circle*{1}}
\put(43,31){\circle*{1}}
\put(43,29){\circle*{1}}
\put(45,30){\line(1,0){20}}
\put(65,30){\circle*{1}}
\put(67,31){\circle*{1}}
\put(67,29){\circle*{1}}

\put(67,31){\line(14,12){4}}
\put(67,29){\line(14,-12){4}}
\put(43,31){\line(-14,12){4}}
\put(43,29){\line(-14,-12){4}}

\put(55,30){\line(0,1){10}}
\put(50,40){\line(2,-3){5}}
\put(65,40){\line(-2,-1){10}}
\put(60,40){\line(0,-1){2.5}}

\put(57.5,25){\line(1,1){2.5}}
\put(62.5,25){\line(-1,1){2.5}}
\put(60,30){\line(0,-1){2.5}}
\put(50,30){\line(0,-1){2}}
\put(46.5,26.5){(empty)}

\put(14.5,40.5){u}
\put(19.5,40.5){j}
\put(23.5,40.9){$t_{\alpha 121}$}
\put(29.5,40.5){q}

\put(49.5,40.5){q}
\put(54.5,40.5){j}
\put(58.9,40.7){$t_{\alpha 121}$}
\put(64.5,40.5){u}

\put(22,23.5){x}
\put(27,23.5){p}
\put(57,23.5){x}
\put(62,23.5){p}


\put(26,15){\line(1,0){13}}
\put(30,15){\line(0,1){4}}
\put(35,15){\line(0,-1){4}}
\put(29.7,19.7){j}
\put(34.5,9.7){x}
\put(41,16){\line(2,2){3}}
\put(41,14){\line(2,-2){3}}
\put(24,16){\line(-2,2){3}}
\put(24,14){\line(-2,-2){3}}
\put(26,15){\circle*{0.7}}
\put(24,14){\circle*{0.7}}
\put(24,16){\circle*{0.7}}
\put(39,15){\circle*{0.7}}
\put(41,16){\circle*{0.7}}
\put(41,14){\circle*{0.7}}

\put(15,65){\circle{3}}
\put(8,70){$\bt_{\alpha 12}$}
\put(57.5,65){\circle{3}}
\put(55,70){$\bt^\prime_{\alpha 12}$}

\end{picture}
\end{figure}


\paragraph{Stage 2.}
In each tree $\bt_{(i)}$ from Stage 1 (branch $i = 1$ is illustrated in Figure \ref{F:construct}, middle), choose a uniform random leaf which is original (not $b_i$ or $t_i$),
 and re-label this as $t_{i1}$ (thereby becoming a novel leaf). 
 The three leaves $b_i, t_i, t_{i1}$ determine a branchpoint.
 Repeat the ``general step" illustrated in Figure \ref{F:construct_1.5}.
 That is,
cut the tree into $3$ branches at that branchpoint, and within each branch  create a new leaf at the cut-point, labeled 
$b_{i1}, b_{i2}, b_{i3}$ ordered as the branches containing $b_i,  t_i, t_{i1}$ respectively. 
Write $B_{ij}$ for the set of original leaf labels in the branch containing $b_{ij}$.
Repeat independently within $\bt^\prime$.
For each $ij$ define $A_{ij} = B_{ij} \cap B^\prime_{ij}$ and assume each $A_{ij}$ is non-empty.
For each $ij$ and each of $\bt$ and $\bt^\prime$ write $\bt_{(ij)}$ and $\bt^\prime_{(ij)}$
for the subtree induced by 
leaf-set $A_{ij} \cup \{b_{ij}, t_{ij}\}$.
This is the Stage 2 configuration.

\paragraph{Stage 3.}
Continue recursively.  Figure \ref{F:construct} (bottom)  shows the part of the
Stage 3 configuration arising from splitting the two trees $\bt_{(11)}$ and $\bt_{(13)}$ arising from Stage 2.

\paragraph{Remark.}
The point of the construction is that after any stage one can
pick one leaf from each non-empty subtree at that stage, and the set of such leaves 
induces a  common subtree within the original two trees $\bt$ and $\bt^\prime$.
So we need some ``stopping rule" -- for a given subtree at a given stage, do we continue to split or do we stop and pick a leaf?
Heuristically it seems optimal to wait until the two subtrees within a branch are identical, though 
if we wait too long then we are liable to get empty subtrees.
For our purpose it suffices to use a simple cutoff rule based on size.

\paragraph{Stage 4.}
We prespecify a ``cut-off size" $K$.
When the number of original leaves in a subtree is less than $K$, do not split.  Instead 
pick arbitrarily one leaf from the subtree, if non-empty.\footnote{Think of $K$ as increasing slowly with $n$.}

Figure  \ref{F:construct_2} shows what happens in detail.
Row 1 indicates, in the style of Figure  \ref{F:construct}, 
a leaf-set $A_{\alpha 12} \cup \{b_{\alpha 1}, b_{\alpha 12} \}$ for
 a tree with more than $K = 10$ original leaves. 
 (Here $\alpha$ denotes an index string.)
The two trees on this leaf-set, $\bt_{\alpha 12}$ and $\bt^\prime_{\alpha 12}$, 
are shown in their entirety in row 2.

We split the trees by picking a random original leaf from each tree; in our example these happen to be $v$ from the left tree and $s$ from the right tree.
Both chosen leaves are re-labeled according to our labeling convention as $t_{\alpha  121}$.
We split each tree at the branchpoint indicated by a circle.
In our example the branches containing $t_{\alpha  121}$ have 6 and 8 original leaves, and the intersection of their leaf-sets is $\{u, j, q, t_{\alpha  121}\}$.
The induced subtrees on that intersection are shown in row 3, along with the induced subtree on $\{x, p\}$ arising  from the branches at $b_{\alpha 12}$.
Note that (in our example) one branch is in fact  empty.
In general $0, 1, 2$ or $3$ branches may be empty, implying that either $3, 2, 1$ or $0$ branches remain.
Because (in our example) all three branches have less than $K = 10$ leaves (in the intersection) we stop splitting and 
pick one leaf from each non-empty branch. 
In Figure \ref{F:construct_2} we picked $j$ and $x$.
The small induced subtree of the picked leaves is shown in row 4.
By construction this is the same subtree within  each original tree.

Note that when a tree is split, it might have less  than $K$ original leaves in some branches and  more than $K$ 
in other branches, in which case we continue to split the ``more than $K$" branches.

\setlength{\unitlength}{0.08in}
\begin{figure}
\caption{The construction: final stage.  Reconnecting the small subtrees into the final common subtree.}
\label{F:construct_3}
\begin{picture}(70,23)
\put(0,10){\circle{0.7}}
\put(9,10){\circle*{0.7}}
\put(0.5,10){\line(1,0){9}}
\put(5,10){\line(0,-1){5}}
\put(4.6,3.7){$\eps$}
\put(11,11){\circle*{0.7}}
\put(11,9){\circle*{0.7}}
\put(11,19){\circle{0.7}}
\put(11,11.5){\line(0,1){7}}
\put(11,13){\line(1,0){3}}
\put(14.5,12.5){$\delta$}
\put(11,17){\line(1,0){3}}
\put(14.5,16.5){$\beta$}
\put(11,15){\line(-1,0){3}}
\put(6.8,14.8){$\gamma$}
\put(14,6){\circle*{0.7}}
\put(11.5,8.5){\line(1,-1){2}}
\put(16,6){\circle*{0.7}}
\put(18.5,8.5){\line(-1,-1){2}}
\put(19,9){\circle*{0.7}}
\put(15,5){\circle*{0.7}}
\put(15,4,5){\line(0,-1){3}}
\put(15,1){\circle{0.7}}
\put(21,10){\circle*{0.7}}
\put(19,11){\circle*{0.7}}
\put(19,19){\circle{0.7}}
\put(19,11.5){\line(0,1){7}}
\put(19,15){\line(1,0){3}}
\put(22.5,14.5){$\alpha$}

\put(21.5,10){\line(1,0){12}}
\put(34,10){\circle*{0.7}}
\put(26,10){\line(0,1){4}}
\put(25.7,14.9){j}
\put(30,10){\line(0,-1){4}}
\put(29.5,4.8){x}
\put(36,11){\circle*{0.7}}
\put(36,9){\circle*{0.7}}
\put(36.5,11.5){\line(1,1){2}}
\put(36.5,8.5){\line(1,-1){2}}

\put(51,10){\line(0,1){9}}
\put(51,13){\line(1,0){3}}
\put(54.5,12.5){$\delta$}
\put(50.5,19.5){$\beta$}
\put(51,15){\line(-1,0){3}}
\put(46.8,14.8){$\gamma$}
\put(46,10){\line(1,0){20}}
\put(45,9.6){$\eps$}

\put(59,10){\line(0,1){4}}
\put(58.7,14.9){j}
\put(62,10){\line(0,-1){4}}
\put(61.5,4.8){x}
\put(55,10){\line(0,-1){4}}
\put(54.5,4.8){$\alpha$}
\put(66,10){\line(1,1){2}}
\put(66,10){\line(1,-1){2}}

\end{picture}
\end{figure}
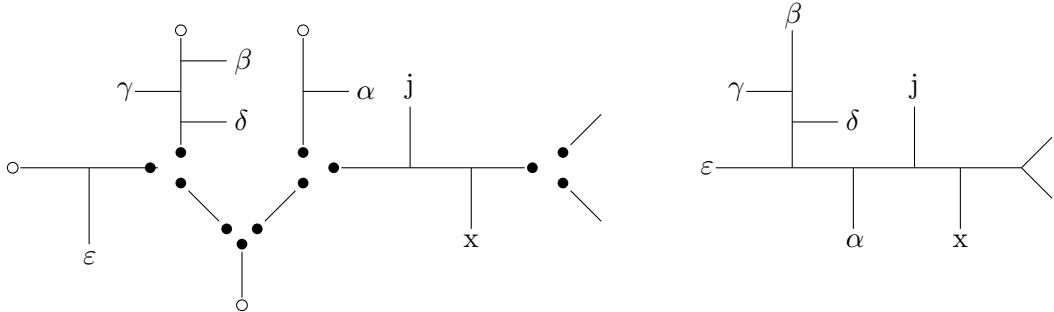

\paragraph{Stage 5.}
When Stage 4 terminates we  have a collection of small subtrees of the following type -- see Figure \ref{F:construct_3}, left, for a part of such a structure. 
Each subtree consists of a distinguished edge, between two branchpoints  $\bullet$ or between a branchpoint and a terminal $\circ$, 
and the  edge has  $0, 1, 2$ or $3$ original  leaves attached to it.

We then take the induced subtree on all these original leaves of that collection (Figure \ref{F:construct_3}, right) to obtain the {\em output common subtree} of the original trees 
$\bt$ and $\bt^\prime$.

\subsection{Analysis of the random construction}
\label{sec:analysis}
The previous section described a construction, that is a randomized algorithm which takes two size-$n$ trees $(\bt_1,\bt_2)$ as input and outputs a common subtree.
When we apply our construction to two independent uniform random trees 
$(\bT_n, \bT^\prime_n)$  on leaf-set $[n]$, then  at the end of each stage, 
there is a random collection (as illustrated in Figure \ref{F:construct}) of leaf-sets of the form $A \cup\{b,t\}$ or $A \cup\{b,b^\prime\}$, 
where $A \subset [n]$ are original leaves and $b,t$ are novel branchpoints or terminals.  
This collection of leaf-sets is the same within the realizations of $\bT_n$ and $ \bT^\prime_n$, though the trees built over these leaf-sets are typically different.
A consequence of the consistency property is that, at the end of Stage 0, 
\begin{quote}
(*) Conditional on the collection, the individual trees on these leaf-sets are all independent and uniform
\end{quote}
and then Proposition \ref{P:key} implies that (*)  holds inductively throughout the construction.
Conceptually, (*) shows  our construction is a ``true recursion", in that we are always considering uniform random trees with two distinguished leaves.

For a given leaf $\ell$ consider
\[ p_{n,K} = \Pr  \mbox{  ($\ell$ in output common subtree) } \]
and note  
\begin{equation}
\Ex \mbox{  (size of output common subtree) } = n p_{n,K}
 \label{size}
 \end{equation}
so it will suffice to lower bound $p_{n,K}$.
We will consider the process $(X_n(t), t = 0, 1, 2 \ldots)$ which records the size of the leaf-set $A$ containing $\ell$ after $t$ stages, with $X_n(t) = 0$ if $\ell$ is not in any Stage-$t$ subtree.
That is, after Stage 1 we have (in each original tree) uniform random trees on the same three leaf-sets $A_1\cup \{b_1,t_1\}, A_2\cup \{b_2,t_2\}, A_3\cup \{b_3,t_3\}$,
and we define
\begin{eqnarray*}
X_n(1) &=&  |A_i| \mbox{ if } \ell \in A_i \\
&=& 0 \mbox{ if } \ell \not\in \cup_{i=1}^3 A_i .
\end{eqnarray*}
A key point is that this process is in fact a
 specific Markov chain $X(t)$ whose transition probabilities (given in the next section) do not depend\footnote{In fact the first step is slightly different (we typically start at $n-4$) but this does not affect the asymptotics.}  on $n$.
By definition, a realization of $(X_n(t), t \ge 0)$ is decreasing and must eventually either make a transition $x \to 0$ from some $x \ge K$, or 
make a transition from some $x \ge K$ to some $x^\prime \in  \{K-1,K-2,\ldots, 1\}$, in other words ``enter $ \{K-1,K-2,\ldots, 1\}$".
Define
\begin{equation}
q_{n,K} := \Pr( \mbox{ $X(t)$ enters } \{K-1,K-2,\ldots, 1\} \ | \ X(0) = n) .
\label{q:def}
\end{equation}
 The ``finally pick arbitrarily one leaf from each non-empty branch" rule implies
\begin{equation}
 q_{n,K}/(K-1) \le p_{n,K} \le q_{n,K} .
 \label{pick}
 \end{equation}
So it will suffice to lower bound $q_{n,K}$.

\paragraph{Outline of approximation method.}
For $X(0) = n$, the scaled process $n^{-1}X(t)$ is approximately the process $Z(t)$ from section \ref{sec:cont},
and so for $\beta = \beta_0 := (\sqrt{3} -1)/2$
the process 
$(Q_n(t) := (n^{-1}X(t))^{\beta - 1}  \ind_{(X(t)>0)} )$ is approximately a martingale.  
If it were exactly a martingale, then applying the optional sampling theorem to 
\begin{equation}
S_K:= \min \{t: \ X(t) \le K-1 \}
\label{SK}
\end{equation}
 we would have
\[ 1 = \Ex [ (n^{-1}X(S_K))^{\beta - 1}  \ind_{(X(S_K) \ge 1)} ] \le n^{1 - \beta} q_{n,K} \]
and so $q_{n,K} \ge n^{\beta - 1}$, giving the desired lower bound via (\ref{pick}) and (\ref{size}).

The argument is formalized in the next two sections.

\subsection{The transition probabilities.}
\label{sec:tp}
First recall that the hypergeometric distribution
\begin{equation}
\Pr(M_{a,a^\prime,m} = j) = \frac{ {a \choose j} {m-a \choose a^\prime -j} }{ {m \choose a^\prime} } , \quad \max(0, a+a^\prime - m)  \le j \le \min(a,a^\prime) 
\label{hyper}
\end{equation}
describes the size of the intersection of a uniform random $a$-subset of $[m]$ with an independent uniform random $a^\prime$-subset of $[m]$,
and 
\[ \Ex  M_{a,a^\prime,m} = \frac{a a^\prime}{m}, \ 
\var M_{a,a^\prime,m} = \frac{a a^\prime (m-a)(m-a^\prime)}{m^2(m-1)} .
\]
We will consider the asymptotic regime
\begin{equation}
 m \to \infty  \mbox{ with } \{a, a^\prime\} \subset [\eps m, (1-\eps) m ] 
\mbox{ for some } \eps > 0.
\label{asy}
\end{equation}
In this regime the normalized quantity $m^{-1} M_{a,a^\prime,m} $ has expectation $\Omega(1)$ and variance $\Omega(1/m)$.
It easily follows that we have  convergence of $\beta$'th moments for $0 < \beta < 1$:
\begin{equation}
\Ex [M^\beta_{a,a^\prime,m} ] = (1 + o(1)) (\Ex  M_{a,a^\prime,m})^\beta \mbox{ in regime  (\ref{asy})} .
\label{moment}
\end{equation}
From the consistency property (*), the event $X(t) = m$ means that at Stage $t$ we have the following property, displayed for clarity.
\begin{quote}
There is some leaf-set, say\footnote{Or  $A \cup \{b,b^\prime\}$.}  $A \cup \{b,t\}$, with $|A| = m$, and 
(within $\bT_n$)
$\ell$ is a uniform random leaf in $A$, in a uniform random induced subtree $\bt^*$ on $A \cup \{b,t\}$, 
and (independently within $\bT^\prime_n$)
$\ell$ is a uniform random leaf in $A$, in a uniform random induced subtree $\bt^{**}$ on $A \cup \{b,t\}$.
\end{quote}
There are 3 possibilities for the next stage.
With probability $1 - (1-m^{-1})^2$, one or both of the randomly chosen terminal leaves will be $\ell$, and then $X(t+1) = 0$.
Otherwise, with the distribution described in Proposition \ref{P:key}, $\bt^*$ is split into three branches on some leaf-sets $A_1, A_2, A_3$ of sizes $a_i = |A_i|$
and independently $\bt^*$ is split into three branches on some leaf-sets $A^\prime_1, A^\prime_2, A^\prime_3$ of sizes $a^\prime_i = |A^\prime_i|$.
Conditionally on these sizes, 
and using the property displayed above,
for each $i$ the size $M_{(i)} := M_{a_i,a_i^\prime,m}$  of the intersection $A_i \cap A^\prime_i$ has the hypergeometric distribution 
(\ref{hyper}), and then (because $\ell$ is a uniform random leaf) 
conditionally on $M_{(i)}$
\begin{equation}
\Pr(X(t+1) = M_{(i)} ) = M_{(i)} /m .
\label{XMM}
\end{equation}
With the remaining probability, that is with probability
$1 - (1-m^{-1})^2 - \sum_{i=1}^3 M_{(i)} /m$, leaf $\ell$ goes into different branches within the two trees and $X(t+1)= 0$.

This implicitly specifies the transition probabilities for the chain $X(t)$.

\subsection{Exploiting the continuous approximation}
\label{sec:formal}

We need to lower bound the hitting probability $q_{n,K}$ at (\ref{q:def}).
For large values of $X(t)$ we can exploit the continuous approximation from section  \ref{sec:cont}.

\begin{Proposition}
\label{P:beta}
For $0 <\beta < 1$, as $m \to \infty$ 
\begin{equation}
\Ex [ (X(t+1) )^{\beta -1}  \ind_{(X(t+1)\ge1)} \vert X(t) = m] \ge (1 - o(1)) \Ex [\sum_{i=1}^3 L_i^\beta] \ m^{\beta -1}
\label{mm0}
\end{equation}
for $L_i = Y_i Y^\prime_i$ as at (\ref{def:L}).
\end{Proposition}
\begin{proof}
Conditional on $(M_{(i)}, 1 \le i \le 3)$, from (\ref{XMM}) the expectation equals
$\sum_i \frac{M_{(i)}}{m} \ ( M_{(i)}) ^{\beta - 1}$.
If we also condition on the sizes $(a_i, a^\prime_i)$ then we find
\[
\Ex [ (X(t+1) )^{\beta -1}  \ind_{(X(t+1)\ge1)} \vert \ |A_i| = a_i, |A^\prime_i| = a^\prime_i, 1 \le i \le 3   ]   \]
\begin{equation}
 =  m^{-1} \Ex [ \sum_i ( M_{(i)})^\beta  \vert \ |A_i| = a_i, |A^\prime_i| = a^\prime_i , 1 \le i \le 3   ]  . 
 \label{MAi}
 \end{equation}
By the ``convergence of moment" result (\ref{moment}) for the hypergeometric distribution of $M_{(i)}$,
for fixed small $\eps > 0$ 
we have, for sufficiently large $m$, that
the quantity (\ref{MAi}) is at least
\[ (1 - \eps) m^{-1} \sum_i ( \Ex[  M_{(i)} \ \vert \ |A_i| = a_i, |A^\prime_i| = a^\prime_i , 1 \le i \le 3  ] )^\beta
\]
on the range
\[ \eps m \le \min  (a_i, a^\prime_i; \ 1 \le i \le 3)  \le  \max (a_i, a^\prime_i; \ 1 \le i \le 3) \le (1 -\eps) m . \]
Because $  \Ex[  M_{(i)} \ \vert \ |A_i| , \  |A^\prime_i| ]  = |A_i| \ |A^\prime_i| /m$
we can take unconditional expectation to get, for sufficiently large $m$,  
\[ \Ex [ (X(t+1) )^{\beta -1}  \ind_{(X(t+1)\ge1)} \vert X(t) = m ] \ge (1- \eps)m^{-1} \sum_i \Ex[ (|A_i(m)| \ |A^\prime_i(m)| /m)^\beta \ind_{G_m} ] \]
where $G_{m}$ is the event  
\[ \{\eps m \le  \min (|A_i(m)| , |A^\prime_i(m))|, 1 \le i \le 3)   \le   \max (|A_i(m), |A^\prime_i(m)|, 1 \le i \le 3) \le (1 -\eps) m \}  \]
and where we write $A_i(m)$ to remember dependence on $m$.
Letting $m \to \infty$ and using
the convergence in distribution (\ref{Dlim}) of $ |A_i(m)| \ |A^\prime_i(m)|/m^2$ to $L_i = Y_iY^\prime_i$, 
\[ \liminf_m 
\Ex [ (X(t+1) )^{\beta -1}  \ind_{(X(t+1)\ge1)} \vert X(t) = m ] \ge (1- \eps)m^{\beta -1} 
\sum_i \Ex [ L_i^\beta \ind_G ]
\]
for
\[G =  \{\eps \le  \min (Y_i, Y^\prime_i ,1 \le i \le 3) \le  \max (Y_i, Y^\prime_i ,1 \le i \le 3) \le (1 -\eps) \} . \]
Finally let $\eps \to 0$ and the assertion of Proposition \ref{P:beta} follows.
\qed
\end{proof}

\subsection{Completing the proof of Theorem \ref{T:1}}
Now fix $\beta < \beta_0$, so that 
$\Ex [\sum_{i=1}^3 L_i^\beta] > 1$.
By Proposition \ref{P:beta}, there exists $K(\beta)$ such that
\[ \Ex [ (X(t+1) )^{\beta -1}  \ind_{(X(t+1)\ge1)} \vert X(t) = m] \ge  \ m^{\beta -1}, \ m \ge K(\beta). \]
This says that the process
\[ ((X(t) )^{\beta -1}  \ind_{(X(t)\ge1)} , \ t \le S_{K(\beta)} )\]
stopped at time 
\[ S_{K(\beta)} := \min \{t: \ X(t) \le K(\beta) -1 \} \]
is a submartingale.
So from the optional sampling theorem
\[ n^{\beta -1} \le \Ex [    (X(S_{K(\beta)}) )^{\beta -1}  \ind_{(X(S_{K(\beta)})\ge1)}    ]
\le q_{n,K(\beta)} .
\]
Combining this with (\ref{pick})  and (\ref{size}),
\begin{equation}
\Ex \mbox{  (size of output common subtree) } \ge n^\beta /K(\beta), \mbox{ for }  n > K(\beta) .
 \label{size2}
 \end{equation}
Theorem \ref{T:1} concerns the maximum size $\kappa(\bT_n,\bT^\prime_n)$ of common subtree,
so
\[ \Ex \kappa(\bT_n,\bT^\prime_n) \ge n^\beta /K(\beta), \mbox{ for }  n > K(\beta) .\]
This holds for each $\beta < \beta_0$, establishing  Theorem \ref{T:1}.

\section{Analogies with the LIS problem}
\label{sec:LIS}
Figure \ref{F1} (bottom left) shows a permutation with an (underlined) increasing subsequence 24578, whose length 5 is the length of the longest  increasing subsequence (LIS)
of that permutation. 
Define the random variable $L_n$ to be the length of the (typically non-unique) LIS of a uniform random permutation of $[n]$.
The monograph \cite{romik} records some of the
extensive known results about $L_n$, 
of which four aspects are noteworthy as background for this article.

\begin{figure}[h!]
\caption{Illustration of LIS (left) and LCS (right)}
\label{F1}
\vspace*{0.1in}
\begin{tabular}{cccccccccccccccccc}
1&\underline{2}&3&\underline{4}&\underline{5}&6&\underline{7}&\underline{8}&9 \quad \quad \quad & 5&\underline{7}&2&\underline{6}&\underline{9}&1&\underline{4}&\underline{8}&3 \\
\underline{2}&9&6&\underline{4}&\underline{5}&3&1&\underline{7}&\underline{8} \quad \quad \quad& \underline{7}&3&1&\underline{6}&\underline{9}&2&5&\underline{4}&\underline{8} 
\end{tabular} 
\end{figure}

\begin{itemize}
\item Showing that the order of magnitude of $\Ex L_n$ is $n^{1/2}$ is very easy: the upper bound by the first moment method (calculating the expected number of long increasing subsequences), 
and the lower bound by picking (if possible) for each $j \le n^{1/2}$ an element $i$ for which 
both $i$ and $\pi(i)$ are in the interval  $[jn^{1/2}, (j+1)n^{1/2}]$. 
\item A reformulation of the LIS question in terms of increasing paths through Poisson points in the plane allows a subadditivity proof of existence 
of a limit $n^{-1/2} \Ex L_n \to c$.
\item More detailed study of $L_n$ leads to a very rich theory \cite{romik}  involving techniques from analysis and algebra as well as combinatorics and probability. 
\item We can re-interpret $L_n$ as the length  of the {\em longest common subsequence} (LCS) for two independent 
 uniform random permutations, illustrated in Figure \ref{F1} (right).  This subsequence  (76948 in the illustration) is found by applying the inverse permutation
 taking the top permutation back to 123456789.
\end{itemize}
As implied by this article, for our common subtrees problem we have not succeeded in completing the first point above (the correct order of magnitude),
and we do not know if there is any ``rich theory" yet undiscovered (the third point).  
And (in contrast to the second point above) we do not see any reformulation that would allow us to use a ``soft" method such as subadditivity to prove existence of a limit exponent $\beta$.
Lastly, the heuristic notion that recursing about centroids is
an ``almost optimal" algorithm, and the fact that branch sizes at the centroid  remain random in the limit, suggests
\begin{Conjecture}
There is a non-degenerate $n \to \infty$ limit distribution for 
$\frac{\kappa(\bT_n,\bT^\prime_n)}{\Ex \kappa(\bT_n,\bT^\prime_n) }$.
\end{Conjecture}
This is very different from the LIS case, where the maximum length is concentrated around its mean.

\subsection{Largest common substructures in probabilistic combinatorics}
\label{sec:LCS}
The final noteworthy point in the previous section suggests a range of {\em largest common substructure} questions that can be asked about a range of random combinatorial 
structures.
Consider the following general setting. 
\begin{quote}
There is a set of $n$ labeled elements $[n] := {1, 2,...,n}$.
There is an instance $S$ of a``combinatorial structure" built over these elements. 
The type of structure is such that for any subset $A \subseteq [n]$  there is an induced substructure of the same type on $A$. 
Given two distinct instances $S_1, S_2$  of the same type of structure on $[n]$, we can ask for each $A \subseteq [n]$  whether the two induced substructures on $A$ are
identical; and so we can define
\[ c(S_1, S_2) =  \max\{ \ |A|: \mbox{ induced substructures on $A$ are identical} \} .\]
 Finally, given a probability distribution $\mu$ on the set of all structures of a particular type, we can consider the random
variable $c(\SS_1, \SS_2)$ 
 where $\SS_1, \SS_2$ are independent random structures with distribution $\mu$.
\end{quote}
The ``common subtrees" setting of this paper (for leaf-labeled binary trees), and the LIS problem for random permutations, both fit this framework.
And so does another well-known result.  
On a general graph, a subset $A$ of vertices defines an induced subgraph, 
so for two graphs $G_1, G_2$ on $[n]$ one can ask for maximum size of $A$ for which the induced subgraphs are identical.
Define the ``coincidence" graph $G_1 * G_2$ to have edges
\[ e \in G_1*G_2 \mbox{ if and only if } (e \in G_1 \mbox{ and } e \in G_2) \mbox { or } (e \not\in G_1 \mbox{ and } e \not\in G_2).\]
Now if $\GG_1, \GG_2$ are independent {E}rd\H{o}s--{R}\'{e}nyi  $G(n, p)$ random graphs, 
this question is just asking for the  maximal clique size of 
$\GG_1* \GG_2$.
But $\GG_1* \GG_2$ is itself the {E}rd\H{o}s--{R}\'{e}nyi  $G(n, q = p^2 + (1-p)^2)$ 
random graph, for which the maximal clique size is a well-understood quantity (\cite{bollobas} section 11.1).

However there is also a fourth setting:  partial orders on the set $[n]$.
As remarked in \cite{me-OP}, for the random partial order obtained from random points in the square with the usual 2-dimensional partial order, 
it is not hard to show that the largest common substructure (partial order) has order $n^{1/3}$.
But neither this, nor other models of random partial orders, have been studied more carefully.

\subsection{Why not recurse at the true centroid?}
\label{sec:whynot}
The construction in section \ref{sec:backheur} -- recursing at the {\em true} centroid -- looks more efficient than our scheme of recursing at a {\em random} centroid.
Alas it is not so simple to analyze.
We would split  the original ``Stage 0" tree at its ``level 0" centroid into three branches, which then become ``Stage  1" trees with a ``root" corresponding to the level 0 centroid.
In constructing a common subtree of Stage 1 trees, we need the common subtree to include the root.
When we split a Stage 1 tree at its level 1 centroid into its branches, which become Stage 2 trees, one of the three trees contains the marked root 
corresponding to the level 0 centroid, but we also need to mark the leaf 
corresponding to the level 1 centroid. 
Then within some Stage 2 trees we need a common subtree constrained to include two marked roots, while others need only one marked root.
The number of constraints increases in further stages, and  seems difficult to analyze rigorously 
 to obtain a better lower bound than our Theorem \ref{T:1}.

However by ignoring those additional constraints we can get an upper bound of order $n^{0.485}$ for the ``recurse at true centroid" scheme", as outlined briefly below. 
It is shown in \cite{me-tri}
that the density function for  the limit normalized branch sizes at the true centroid is 
\begin{equation}
\phi(x_1,x_2,x_3) = \frac{1}{12 \pi}  \prod_i x_i^{-3/2}  \mbox{ \ \ on } 
\{(x_1,x_2,x_3) : x_i > 0, \sum_i x_i = 1, \max x_i < 1/2\} .
\end{equation}
One can now solve equation
(\ref{eq:heur}) numerically to get $\beta = 0.485.....$.
This could be one approach to proving Open Problem \ref{OP:1}, if one could somehow quantify the intuition that this 
``recurse at true centroid" scheme is sufficiently close to optimal.

\paragraph{Acknowledgements.} I thank Jean Bertoin and Mike Steel for helpful references,
and two anonymous referees for helpful comments.


\begin{thebibliography}{10}

\bibitem{me-recursive}
David~J.  Aldous.
\newblock Recursive self-similarity for random trees, random triangulations and
  {B}rownian excursion.
\newblock {\em Ann. Probab.}, 22(2):527--545, 1994.

\bibitem{me-tri}
David~J.  Aldous.
\newblock Triangulating the circle, at random.
\newblock {\em Amer. Math. Monthly}, 101(3):223--233, 1994.

\bibitem{me-yule}
David~J. Aldous.
\newblock Stochastic models and descriptive statistics for phylogenetic trees,
  from {Y}ule to today.
\newblock {\em Statist. Sci.}, 16(1):23--34, 2001.

\bibitem{me-OP}
David~J. Aldous.
\newblock Open problem: Largest common substructures in probabilistic
  combinatorics, Originally posted 2003.
\newblock
  https://www.stat.berkeley.edu/$\sim$aldous/Research/OP/substructures.html.

\bibitem{athreya}
Krishna~B. Athreya and Peter~E. Ney.
\newblock {\em Branching processes}.
\newblock Springer-Verlag, New York-Heidelberg, 1972.
\newblock Die Grundlehren der mathematischen Wissenschaften, Band 196.

\bibitem{bernstein}
Daniel~Irving Bernstein, Lam Si~Tung Ho, Colby Long, Mike Steel, Katherine
  St.~John, and Seth Sullivant.
\newblock Bounds on the expected size of the maximum agreement subtree.
\newblock {\em SIAM J. Discrete Math.}, 29(4):2065--2074, 2015.

\bibitem{bertoin}
Jean Bertoin and Servet Mart\'{\i}nez.
\newblock Fragmentation energy.
\newblock {\em Adv. in Appl. Probab.}, 37(2):553--570, 2005.

\bibitem{blum}
Michael G.~B. Blum and Olivier François.
\newblock {Which Random Processes Describe the Tree of Life? A Large-Scale
  Study of Phylogenetic Tree Imbalance}.
\newblock {\em Systematic Biology}, 55(4):685--691, 08 2006.

\bibitem{bollobas}
B\'{e}la Bollob\'{a}s.
\newblock {\em Random graphs}, volume~73 of {\em Cambridge Studies in Advanced
  Mathematics}.
\newblock Cambridge University Press, Cambridge, second edition, 2001.

\bibitem{extremal}
Magnus Bordewich, Simone Linz, Megan Owen, Katherine~St. John, Charles Semple,
  and Kristina Wicke.
\newblock On the maximum agreement subtree conjecture for balanced trees, 2020.
\newblock https://arxiv.org/abs/2005.07357.

\bibitem{corwin}
Ivan Corwin.
\newblock Commentary on ``{L}ongest increasing subsequences: from patience
  sorting to the {B}aik-{D}eift-{J}ohansson theorem'' by {D}avid {A}ldous and
  {P}ersi {D}iaconis.
\newblock {\em Bull. Amer. Math. Soc. (N.S.)}, 55(3):363--374, 2018.

\bibitem{evans}
Steven~N. Evans.
\newblock {\em Probability and real trees}, volume 1920 of {\em Lecture Notes
  in Mathematics}.
\newblock Springer, Berlin, 2008.
\newblock Lectures from the 35th Summer School on Probability Theory held in
  Saint-Flour, July 6--23, 2005.

\bibitem{misra}
Pratik Misra and Seth Sullivant.
\newblock Bounds on the expected size of the maximum agreement subtree for a
  given tree shape.
\newblock {\em SIAM J. Discrete Math.}, 33(4):2316--2325, 2019.

\bibitem{dirichlet}
Kai~Wang Ng, Guo-Liang Tian, and Man-Lai Tang.
\newblock {\em Dirichlet and related distributions}.
\newblock Wiley Series in Probability and Statistics. John Wiley \& Sons, Ltd.,
  Chichester, 2011.
\newblock Theory, methods and applications.

\bibitem{parberry}
Ian Parberry.
\newblock {\em Problems on algorithms}.
\newblock Prentice Hall, Inc., Englewood Cliffs, NJ, 1995.

\bibitem{pittel}
Boris Pittel.
\newblock Expected number of induced subtrees shared by two independent copies of the terminal tree in a critical branching process.
\newblock arXiv 2102.05852.

\bibitem{purvis}
Andy Purvis, Susanne~A. Fritz, Jes\'{u}s Rodríguez, Paul~H. Harvey, and
  Richard Grenyer.
\newblock The shape of mammalian phylogeny: patterns, processes and scales.
\newblock {\em Philosophical Transactions of the Royal Society B: Biological
  Sciences}, 366(1577):2462--2477, 2011.

\bibitem{romik}
Dan Romik.
\newblock {\em The surprising mathematics of longest increasing subsequences},
  volume~4 of {\em Institute of Mathematical Statistics Textbooks}.
\newblock Cambridge University Press, New York, 2015.

\bibitem{rosler}
U.~R\"{o}sler.
\newblock On the analysis of stochastic divide and conquer algorithms.
\newblock {\em Algorithmica}, 29(1-2):238--261, 2001.
\newblock Average-case analysis of algorithms (Princeton, NJ, 1998).

\bibitem{steel1}
Charles Semple and Mike Steel.
\newblock {\em Phylogenetics}, volume~24 of {\em Oxford Lecture Series in
  Mathematics and its Applications}.
\newblock Oxford University Press, Oxford, 2003.

\bibitem{steel2}
Mike Steel.
\newblock {\em Phylogeny---discrete and random processes in evolution},
  volume~89 of {\em CBMS-NSF Regional Conference Series in Applied
  Mathematics}.
\newblock Society for Industrial and Applied Mathematics (SIAM), Philadelphia,
  PA, 2016.

\bibitem{xue}
Chi Xue, Zhiru Liu, and Nigel Goldenfeld.
\newblock Scale-invariant topology and bursty branching of evolutionary trees
  emerge from niche construction.
\newblock {\em Proceedings of the National Academy of Sciences},
  117:7879--7887, 2020.

\bibitem{yukich}
Joseph~E. Yukich.
\newblock {\em Probability theory of classical {E}uclidean optimization
  problems}, volume 1675 of {\em Lecture Notes in Mathematics}.
\newblock Springer-Verlag, Berlin, 1998.

\end{thebibliography}

\end{document}